\documentclass[11pt,dvips,twoside,letterpaper]{article}
\usepackage{pslatex}
\usepackage{fancyhdr}
\usepackage{graphicx}
\usepackage{geometry}
 \RequirePackage[T1]{fontenc}

\def\figurename{Figure} 
\makeatletter
\renewcommand{\fnum@figure}[1]{\figurename~\thefigure.}
\makeatother

\def\tablename{Table} 
\makeatletter
\renewcommand{\fnum@table}[1]{\tablename~\thetable.}
\makeatother

\usepackage{amsmath}
\usepackage{amssymb}
\usepackage{amsfonts}
\usepackage{amsthm,amscd}

\newtheorem{theorem}{Theorem}[section]
\newtheorem{lemma}[theorem]{Lemma}
\newtheorem{corollary}[theorem]{Corollary}
\newtheorem{proposition}[theorem]{Proposition}
\theoremstyle{definition}
\newtheorem{definition}[theorem]{Definition}

\theoremstyle{remark}
\newtheorem{remark}[theorem]{Remark}

\numberwithin{equation}{section}

\def\P{\mathbb P}

\def\R{\mathbb R}
\def\E{\mathbb E}

\def\E{\mathbb E}
\def\N{\mathbb N}

\begin{document}
\title{ $L^{p}$-solutions of backward doubly stochastic differential
equations\thanks{The work is supported by TWAS Research Grants to individuals (No. 09-100 RG/MATHS/AF/AC-I-
UNESCO FR: 3240230311)} }
\author{Auguste Aman\thanks{augusteaman5@yahoo.fr, auguste.aman@univ.cocody.ci}\\
{\it U.F.R Mathématiques et informatique, Universit\'{e} de Cocody},\\ {\it 582 Abidjan 22, C\^{o}te d'Ivoire}
}

\date{}
\maketitle

\begin{abstract}
The goal of this paper is to solve backward doubly
stochastic differential equations (BDSDEs, in short) under weak assumptions on the data. The first part is devoted to the
development of some new technical aspects of stochastic calculus related to this BDSDEs. Then we
derive a priori estimates and prove existence and uniqueness of solution in $L^{p}$, $p\in(1,2)$, extending the work of Pardoux and Peng (see Probab. Theory Related Fields 98 (1994), no. 2).
\end{abstract}

\textbf{MSC 2000:} 60H05,\, 60H20.\\
\textbf{Key Words}: Backward doubly stochastic differential equations; monotone generator; $p$-integrable data.
\section{Introduction}
In this paper, we are concerned with backward doubly stochastic differential equations
(BDSDEs, in short):
\begin{eqnarray*}
Y_t &=& \xi +\int^{ T}_{t}f(r, Y_r, Z_r) dr+\int_{t}^{T}g(r,Y_r,Z_r)\overleftarrow{dB}_r-\int^{T}_{t}Z_r dW_r ,\;  0\leq t\leq T,\label{a0}
\end{eqnarray*}
which involves both a standard forward stochastic Itô integral driven by $dW_t$ and
a backward stochastic Kunita-Itô integral driven by $\overleftarrow{dB}_t$ (see, Kunita \cite{Kunita}). The random variable $\xi$ and functions $f$ and $g$ are data, while the pair of
 processes $(Y_t,Z_t)_{t\in[0,T]}$ is the unknown.

The theory of nonlinear backward doubly SDEs have been firstly introduced in \cite{PP1} by Pardoux and Peng . They show that, under Lipschitz conditions on $f$ and $g$, the above "backward doubly SDE" has a unique solution. Next, in the Markovian framework, BDSDEs is coupled with the forward SDE as follows: for each $(t,x)\in[0,T]\times\R^d$
\begin{eqnarray}
X^{t,x}_s&=&x+\int_t^s b(X^{t,x}_r)dr+\int_t^s \sigma(X^{t,x}_r)dW_r\label{a0'}\\
Y^{t,x}_s &=& l(X^{t,x}_T) +\int^{ T}_{s}f(r,X^{t,x}_r, Y^{t,x}_r, Z^{t,x}_r) dr+\int_{s}^{T}g(r,X^{t,x}_r,Y^{t,x}_r,Z^{t,x}_r)\overleftarrow{dB}_r-\int^{T}_{t}Z^{t,x}_r dW_r,\;\; t\leq s\leq T.\nonumber
\end{eqnarray}
Let $\{(X^{t,x}_s,Y^{t,x}_s,Z^{t,x}_s);\;\; t\leq s\leq T\}$ be the solution of \eqref{a0'}. Under stronger conditions on the coefficients ($b,\,\sigma,\, l,\, f,\, g$ are $C^3$) they proved that
\begin{eqnarray}
u(t,x)=Y^{t,x}_t\label{a0"}
\end{eqnarray}
is the unique classical solution of the quasi-linear stochastic partial differential equations (SPDEs)
\begin{eqnarray*}
SPDE {(f,g)}\left\{
\begin{array}{l}
(i)\,\,\displaystyle{du(t,x)+[
Lu(t,x)+f(t,x,u(t,x),\sigma^{*}(x)D_{x}u(t,x))]dt}\\\displaystyle{\;\;\;\;\;\;\;\;\;\;\;\
\;\;\;\;\;\;\;\;+g(t,x,u(t,x),\sigma^{*}(x)D_{x}u(t,x))\overleftarrow{dB}_{s}=0,\,\,\
(t,x)\in[0,T]\times\R^d}, \\\\
(ii)\,\,u(T,x)=l(x),\,\,\,\,\,\,\ x\in\R^d,
\end{array}\right.
\end{eqnarray*}
where $L$, the operator infinitesimal of the diffusion $X$, is defined by
\begin{eqnarray*}
L=\frac{1}{2}\sum_{i,j=1}^{d}(\sigma(x)\sigma^{*}(x))_{i,j}
\frac{\partial^{2}}{\partial x_{i}\partial
x_{j}}+\sum^{d}_{i=1}b_{i}(x)\frac{\partial}{\partial x_{i}},\quad
\forall\, x\in\R^d.
\end{eqnarray*}
The relation \eqref{a0"} generalizes the well-know Feymann-Kac formula to SPDEs.

After the first existence and uniqueness result established in \cite{PP1}, many other works have been
devoted to
existence and/or uniqueness results for BDSDEs under weaker assumptions on the coefficient $f$. For scalar BDSDEs case, N'zi and Owo \cite{NO} deal with discontinuous coefficients
by using the comparison theorem establish in \cite{Shial}. There is  no comparison theorem for multidimensional BDSDEs. To overcome this difficulty, a monotonicity assumption on the generator $f$ with respect $y$ uniformly on $z$ is used. This condition appears in Peng and Shi \cite{PengShi}, N'zi and Owo \cite{NO1} and Aman and Owo \cite{AO}.

However, in all above works the data are supposed to be at least square integrable. This condition is too restrictive to be assumed in many applications.
For example, American claim pricing is equivalent to solve the following linear BDSDE
\begin{eqnarray}
-dY_t=(r_tY_t+\theta_t Z_t)dt+c_tY_t\overleftarrow{dB}_t-Z_tdW_t,\;\; Y_T=\xi,\label{linear}
\end{eqnarray}
where $r_t$ is the interest rate, $\theta_t$ is the risk premium vector and $c_t$ is the volatility of the exterior effect to the market. Generally, all this coefficients are not necessary bounded and the terminal condition $\xi$ can be only integrable. Therefore, all above works are not usable in this case.

The aim of this paper is to correct this gap and prove existence and uniqueness result for $d$-dimensional BDSDEs under the $p$-integrable condition on $\xi,\; f(t,0,0)$ and $g(t,0,0)$ for any $p\in (1,2)$ and the monotonic condition on $f$. To our knowledge, this result do not exist in literature, therefore it is new. Let us remark that in two previous works, we have already derived, in \cite{A1} and \cite{A2}, the existence and uniqueness result for $L^p$-solution for reflected generalized BSDEs.

The paper is organized as follows. In Section 2, we give all notations and
basic identities of this paper. The Section 3 contains essential a priori estimates. In Section 4, we prove existence and uniqueness result.

\section{ Preliminaries}
\setcounter{theorem}{0} \setcounter{equation}{0}
\subsection{Assumptions and basic notations}
Let $\R^{d\times d}$ be identified to the space of real matrices with $d$ rows and $d$ columns; hence for each $z\in\R^{d\times d},\;|z|^2=trace(zz^*)$.

In throughout this paper, we consider the probability space $(\Omega, \mathcal{F}, \P)$ and $T$ a real and positive constant.
We define on $(\Omega, \mathcal{F}, \P)$ two mutually independent standard Brownian motion processes $\{W_t, 0\leq t\leq T\}$ and $\{B_t, 0\leq t\leq T\}$
taking values in $\R^{d}$ and $\R^{\ell}$ respectively. Let $\mathcal{N}$ denote the class of $\P$-null sets of $\mathcal{F}$ and define
\begin{eqnarray*}
\mathcal{F}_{t}=\mathcal{F}_{t,T}^{B}\vee\mathcal{F}^{W}_{t}\vee\mathcal{N},\;\; 0\leq t\leq T
\end{eqnarray*}
where by $\mathcal{F}^{\eta}_{s,t}=\sigma\{\eta_{r}-\eta_{s},s\leq r\leq t\}$ for any $\eta_t$, and $\mathcal{F}^{\eta}_t=\mathcal{F}^{\eta}_{0,t}$.

We emphasize that the collection $\{\mathcal{F}_t, t\in[0, T]\}$ is not a filtration. Indeed, it is neither increasing nor decreasing.

Next, for any real $p>0$, we denote by $\mathcal{S}^{p}(\R^{n})$ the set of jointly measurable processes $
 \{X_{t}\}_{t\in \lbrack 0,T]}$ taking values in $\R^{n}$ such that
\begin{description}
\item $(i)$
\begin{equation*}
\|X\|_{\mathcal{S}^{p}}=\E\left( \sup\limits_{0\leq t\leq
T}|X_{t}|^{p}\right) ^{1\wedge \frac{1}{p}}<+\infty;
\end{equation*}
\item $(ii)$\,$X_t$  is $\mathcal{F}_t$-measurable, for any $t\in[0,T]$.
\end{description}
and $\mathcal{M}^{p}(\R^{n})$) the set of (classes of $d\P\otimes dt$ a.e.
equal) $n$-dimensional jointly measurable processes such that
\begin{description}
\item $(i)$
\begin{eqnarray*}
\|X\|_{\mathcal{M}^{p}}=\E \left[ \left(
\int_{0}^{T}|X_{t}|^{2}dt\right) ^{\frac{p}{2}}\right] ^{1\wedge \frac{1}{p}%
}<+\infty .
\end{eqnarray*}
\item $(ii)$\,$X_t$  is $\mathcal{F}_t$-measurable,\, for a.e. $t\in[0, T]$.
\end{description}
If $p\geq 1$, $(\mathcal{S}^{p}(\R^n),\,\|X\|_{\mathcal{S}^{p}})$ and $(\mathcal{M}%
^{p}(\R^{n}),\,\|X\|_{\mathcal{M}^{p}})$ are Banach spaces.

Let
\begin{eqnarray*}
f:\Omega\times [0,T]\times\R^k
\times\R^{d\times d }\rightarrow \R^d; \ g:\Omega\times [0,T]\times\R^d
\times\R^{d\times d }\rightarrow \R^{d\times\ell}
\end{eqnarray*}
be jointly measurable such that for any $(y,z)\in \R^d\times\R^{d\times d }$. We have
\begin{description}
\item ({\bf H1})\, $f(.,y,z)\in \mathcal{M}^{p}(0,T,\R^d),\,\,\ g(.,y,z)\in\mathcal{M}^{p}(0,T,\R^{d\times\ell})$
\end{description}
\begin{description}
\item $({\bf H2})$
There exists constants $\mu\in\R,\;\lambda>0$  and
$0<\alpha<1$ such that for any $t\in[0,T];\,
(y_{1},z_{1}), (y_{2},z_{2})\in \R^d\times\R^{d\times d}$,\newline
$
\left\{
\begin{array}{l}
(i)\; |f(t,y_{1},z_{1})-f(t,y_{1},z_{2})|\leq\lambda|z_{1}-z_{2}|,\\\\
(ii)\;\langle y_1-y_2,f(t,y_1,z_1)-f(t,y_2,z_1)\rangle\leq\mu|y_1-y_2|^{2},\\\\
(iii)\;|g(t,y_{1},z_{1})-g(t,y_{2},z_{2})|^{2}\leq\lambda|y_{1}-y_{2}|^{2}+\alpha|z_{1}-z_{2}|^{2}.
\end{array}\right.
$
\end{description}
Given a $\R^d$-valued $\mathcal{F}_{T}$-measurable random vector $\xi$,
we consider the backward doubly stochastic
differential equation:
\begin{eqnarray}
Y_{t}=\xi+\int_{t
}^{T}f(s,Y_{s},Z_{s})ds+\int_{t}^{T}g(s,Y_{s},Z_{s})\,\overleftarrow{dB}_{s}-\int_{t}^{
T}Z_{s}\,dW_{s},\,\ 0\leq t\leq T.\label{a1}
\end{eqnarray}
Now we recall what we mean by a solution to the BDSDE \eqref{a1}.
\begin{definition}
A solution of BDSDE $(\ref{a1})$ is a pair $(Y_{t},Z_{t})_{0\leq t \leq
T}$  of jointly measurable processes taking values in $\R^{k}
\times \R^{k\times d}$ and satisfying $(\ref{a1})$ such that: $\P$ a.s., $t\mapsto (Z_t,g(t,Y_t,Z_t))$ belongs in $L^{2}(0,T),\,
t\mapsto f(t,Y_t,Z_t)$ belongs in $L^{1}(0,T)$.
\end{definition}

\subsection{Generalized Tanaka formula}
As explained in the introduction, we want to deal with BDSDEs with data in $L^{p}$, $p\in(1,2)$ like the works of Briand et al. (see \cite{Pal}) which treat BSDEs case i.e $g\equiv 0$.  We start by Tanaka formula relative to BDSDEs, which is the critical tool in this paper. For this, we note
$\hat{x} = |x|^{-1}x{\bf 1}_{\{x\neq0\}}$.
\begin{lemma}
Let $\{K_t\}_{t\in[0,T]},\,\{H_t\}_{t\in[0,T]}$ and $\{G_{t}\}_{t\in[0,T]}$ be jointly measurable such that $K\in \mathcal{M}^{p}(0,T,\R^d),\,\,\ H\in\mathcal{M}^{p}(0,T,\R^{k\times d}),\;\; G\in\mathcal{M}^{p}(0,T,\R^{d\times\ell})$.
We consider the $\R^{d}$-valued semi martingale $\{X_t\}_{t\in[0,T]}$ defined by
\begin{eqnarray}
X_{t}=X_0+\int_{0}^{t}K_s\ ds+\int_{0}^{t}G_s\ \overleftarrow{dB}_{s}
+\int_{0}^{t}H_s\ dW_{s},\;\; 0\leq t\leq T.
\end{eqnarray}
Then, for any $p\geq 1$, we have
\begin{eqnarray*}
|X_{t}|^{p}-{\bf 1}_{\{p=1\}}L_t&=&|X_0|^{p}+p\int_{0}^{t}|X_{s}|^{p-1}\langle\hat{X}_{s},K_s\rangle ds\\
&&+p\int_{0}^{t}|X_{s}|^{p-1}\langle \hat{X}_{s},G_{s}\overleftarrow{dB}_{s}\rangle+p\int_{0}^{t}|X_{s}|^{p-1}\langle \hat{X}_{s},H_{s}dW_{s}\rangle\\
&&-\frac{p}{2}\int_{0}^{t}|X_{s}|^{p-2}{\bf 1}_{\{X_s\neq 0\}}\{(2-p)(|G_s|^{2}-\langle \hat{X}_{s},G_sG_s^{*}\hat{X}_{s}\rangle)+(p-1)|G_s|^{2}\}ds\\
&&+\frac{p}{2}\int_{0}^{t}|X_{s}|^{p-2}{\bf 1}_{\{X_s\neq 0\}}\{(2-p)(|H_s|^{2}-\langle \hat{X}_{s},H_sH_s^{*}\hat{X}_{s}\rangle)+(p-1)|H_s|^{2}\}ds,
\end{eqnarray*}
where $\{L_t\}_{t\in[0,T]}$ is a continuous process with $L_0=0$, which varies only
on the boundary of the random set $\{t\in[0,T], \, X_t = 0\}$.
\end{lemma}
\begin{proof}
Since the function $x\mapsto|x|^p$ is not smooth enough, for $p\in[1,2)$, we approximate it by the function $
\begin{array}{l}
u_{\varepsilon}(x)=(|x|^{2}+\varepsilon^{2})^{1/2},\;\; \forall\;\;\varepsilon >0,
\end{array}
$
which is actually a smooth function. We have, denoting $I$ the identity
matrix of $\R^{d\times d}$,
\begin{eqnarray*}
\nabla u_{\varepsilon}^{p}(x)=pu_{\varepsilon}^{p-2}(x)x;\;\;\; D^2u_{\varepsilon}^{p}(x)=pu^{p-2}_{\varepsilon}(x)I+p(p-2)u^{p-4}_{\varepsilon}(x)(x\otimes x)
\end{eqnarray*}
such that Itô's formula leads
\begin{eqnarray}
u_{\varepsilon}^{p}(X_{t})&=&u^{p}_{\varepsilon}(X_0)+p\int_{0}^{t}u_{\varepsilon}^{p-2}(X_{s})\langle X_{s},K_s\rangle ds\nonumber\\
&&+p\int_{0}^{t}u_{\varepsilon}^{p-2}(X_{s})\langle X_{s},G_{s}\overleftarrow{dB}_{s}\rangle+p\int_{0}^{t}u_{\varepsilon}^{p-2}(X_{s})\langle X_{s},H_{s}dW_{s}\rangle\nonumber\\
&&-\frac{1}{2}\int_{0}^{t}trace(D^{2}u_{\varepsilon}^{p}(X_{s})G_sG_s^{*})ds+\frac{1}{2}\int_{0}^{t}trace(D^{2}u_{\varepsilon}^{p}(X_{s})H_sH_s^{*})ds.\label{Tanaka}
\end{eqnarray}
It remains essentially to pass to the limit when $\varepsilon\rightarrow 0$ in \eqref{Tanaka}. To do this, let remark first that
\begin{eqnarray*}
\int_{0}^{t}u_{\varepsilon}^{p-2}(X_{s})\langle X_{s},K_s\rangle ds\rightarrow\int_{0}^{t}|X_{s}|^{p-1}\langle \widehat{X}_{s},K_s\rangle ds,\;\; \P\mbox{-a.s}.
\end{eqnarray*}
We also have
\begin{eqnarray*}
\int_{0}^{t}u_{\varepsilon}^{p-2}(X_{s})\langle X_{s},G_s\overleftarrow{dB}_s\rangle\rightarrow\int_{0}^{t}|X_{s}|^{p-1}\langle
 \widehat{X}_{s},G_s\overleftarrow{dB}_s\rangle
\end{eqnarray*}
and
\begin{eqnarray*}
\int_{0}^{t}u_{\varepsilon}^{p-2}(X_{s})\langle X_{s},H_sdW_s\rangle\rightarrow\int_{0}^{t}|X_{s}|^{p-1}\langle \widehat{X}_{s},H_sdW_s\rangle
\end{eqnarray*}
in $\P$-probability uniformly on $[0, T]$. The convergence of the stochastic integrals follows from the following convergence:
\begin{eqnarray*}
\int_{0}^{T}|X_s|^2{\bf 1}_{\{X_s\neq 0\}}|G_s|^2(|X_{s}|^{p-2}-u_{\varepsilon}^{p-2}(X_{s}))^2ds\rightarrow 0
\end{eqnarray*}
and
\begin{eqnarray*}
\int_{0}^{T}|X_s|^2{\bf 1}_{\{X_s\neq 0\}}|H_s|^2(|X_{s}|^{p-2}-u_{\varepsilon}^{p-2}(X_{s}))^2ds\rightarrow 0,
\end{eqnarray*}
which is provided by the dominated convergence theorem.

It remains to study the convergence of the term including the second derivative of
$u_{\varepsilon}$. It is shown in \cite{Pal} that
\begin{eqnarray*}
trace(D^{2}u_{\varepsilon}^{p}(X_{s})G_sG_s^{*})&=& p(2-p)(|X_s|u^{-1}_{\varepsilon}(X_{s}))^{4-p}|X_s|^{p-2}{\bf 1}_{\{X_s\neq 0\}}(|G_s|^2-\langle\widehat{X}_s,G_sG^*_s\widehat{X}_s\rangle)\\
&&+p(p-1)(|X_s|u^{-1}_{\varepsilon}(X_{s}))^{4-p}|X_s|^{p-2}{\bf 1}_{\{X_s\neq 0\}}|G_s|^2+p\varepsilon^2|G_s|^2u^{p-4}_{\varepsilon}(X_{s})
\end{eqnarray*}
and
\begin{eqnarray*}
trace(D^{2}u_{\varepsilon}^{p}(X_{s})H_sH_s^{*})&=& p(2-p)(|X_s|u^{-1}_{\varepsilon}(X_{s}))^{4-p}|X_s|^{p-2}{\bf 1}_{\{X_s\neq 0\}}(|H_s|^2-\langle\widehat{X}_s,H_sH^*_s\widehat{X}_s\rangle)\\
&&+p(p-1)(|X_s|u^{-1}_{\varepsilon}(X_{s}))^{4-p}|X_s|^{p-2}{\bf 1}_{\{X_s\neq 0\}}|H_s|^2+p\varepsilon^2|H_s|^2u^{p-4}_{\varepsilon}(X_{s}).
\end{eqnarray*}
One has also
\begin{eqnarray}
|G_s|^2 &\geq& \langle \widehat{X}_s,G_sG_s^*\widehat{X}_s\rangle \nonumber\\
|H_s|^2 &\geq& \langle \widehat{X}_s,H_sH_s^*\widehat{X}_s\rangle\label{ingd}
\end{eqnarray}
and $$\frac{|X_s|}{u_{\varepsilon}(X_s)}\nearrow {\bf 1}_{\{X_s\neq 0\}}$$ as $\varepsilon\rightarrow 0$. Hence by monotone convergence, as $\varepsilon\rightarrow 0$,
\begin{eqnarray*}
\int_0^t(|X_s|u^{-1}_{\varepsilon}(X_{s}))^{4-p}|X_s|^{p-2}{\bf 1}_{\{X_s\neq 0\}}\{(2-p)(|G_s|^2-\langle\widehat{X}_s,G_sG^*_s\widehat{X}_s\rangle)+(p-1)|G_s|^2\}ds
\end{eqnarray*}
converges to
\begin{eqnarray*}
\int_0^t|X_s|^{p-2}{\bf 1}_{\{X_s\neq 0\}}\{(2-p)(|G_s|^2-\langle\widehat{X}_s,G_sG^*_s\widehat{X}_s\rangle)+(p-1)|G_s|^2\}ds
\end{eqnarray*}
and
\begin{eqnarray*}
\int_0^t(|X_s|u^{-1}_{\varepsilon}(X_{s}))^{4-p}|X_s|^{p-2}{\bf 1}_{\{X_s\neq 0\}}\{(2-p)(|H_s|^2-\langle\widehat{X}_s,H_sH^*_s\widehat{X}_s\rangle)+(p-1)|H_s|^2\}ds
\end{eqnarray*}
converges to
\begin{eqnarray*}
\int_0^t|X_s|^{p-2}{\bf 1}_{\{X_s\neq 0\}}\{(2-p)(|H_s|^2-\langle\widehat{X}_s,H_sH^*_s\widehat{X}_s\rangle)+(p-1)|H_s|^2\}ds,
\end{eqnarray*}
$\P$-a.s., for all $0\leq t\leq T$.

Let denote
\begin{eqnarray*}
L_t^{\varepsilon}(p)=\int_0^tC_s^{\varepsilon}(p)ds,
\end{eqnarray*}
where $C_s^{\varepsilon}(p)=\frac{p}{2}\varepsilon^2u^{p-4}_{\varepsilon}(X_{s})(|H_s|^2-|G_s|^2)$.
Then it follows from \eqref{Tanaka} that $L^{\varepsilon}(p)$ converges to a continuous process $L(p)$ as $\varepsilon\rightarrow 0$ such that $L(p)\equiv 0$ for $p>1$. Indeed, for $p\geq 4,\, L(p)\equiv 0$ since $C_s^{\varepsilon}(p)$ converges to $0$ in $L^1(0, T)$.
Next, if $p\in(1, 4)$, by setting $\theta= (4-p)/3\in(0, 1)$ we get
\begin{eqnarray*}
L^{\varepsilon}_t=\frac{p}{2}\int_0^t\left(\varepsilon^2(|H_s|^2-|G_s|^2)u^{-3}_{\varepsilon}(X_{s})\right)^{\theta}\left(\varepsilon^2(|H_s|^2-|G_s|^2)\right)^{1-\theta}ds.
\end{eqnarray*}
Hence Hölder's inequality provides
\begin{eqnarray*}
L_t^{\varepsilon}(p)\leq p(L_t^{\varepsilon}(1))^{\theta}\left(\int^{T}_{0}\varepsilon^2(|H_s|^2-|G_s|^2)ds\right)^{1-\theta}
\end{eqnarray*}
which tends to $0$ as $\varepsilon\rightarrow 0$ for each $t\in[0,T]$.

For $p=1$, let set $L(1)=L$ and denote $A=\{t\in[0,T],\; X_t=0\}$. If $t$ belongs in the interior of $A$, one can find $\delta>0$ such that $X_s=0$ whenever $|t-s|\leq \delta$. Therefore, the quadratic variation of $X$ is constant on  $[t-\delta,t+\delta]$ and then $H_s=\pm G_s$ almost everywhere on this interval. On the other hand, if $t$ belongs in the complement of the set $A$, there exits $\delta>0$ such that $X_s\neq 0$ when $|t-s|\leq \delta$. In both cases, $C^{\varepsilon}(1)$ converges to $0$ in $L^1(t-\delta,t+\delta)$ and
\begin{eqnarray*}
L_{s}-L_{r}=\lim_{\varepsilon \rightarrow 0} \int^{s}_{r}C^{\varepsilon}_s(1)ds=0,\; \forall\; s,r\in(t-\delta,t+\delta).
\end{eqnarray*}
Therefore $L_t$ is neither increasing nor decreasing and  varies only on the boundary on $A$.
This concludes the proof of the lemma.
\end{proof}
\begin{remark}
Since the process $L$ is neither increasing nor decreasing, we can not apply the similarly argument used in \cite{Pal}. Therefore
the following corollary works only in the case $p\in(1,2)$, which correspond to our framework.
\end{remark}
\begin{corollary}
Let $p\in(1,2)$ and denote $c(p)=p(p-1)/2$ and $\bar{c}(p)=p(3-p)/2$ . If $(Y,Z)$ is a solution of the BDSDE \eqref{a1}, then for $0\leq t\leq T$
\begin{eqnarray*}
&&|Y_{t}|^{p}+c(p)\int_{t}^{T}|Y_s|^{p-2}{\bf 1}_{\{Y_s\neq 0\}}|Z_s|^{2}ds\\
&\leq &|Y_T |^{p}+p\int_{t}^{T}|Y_{s}|^{p-1}\langle\hat{Y}_{s},f(s,Y_s,Z_s)\rangle ds\\
&&+\bar{c}(p)\int_{t}^{T}|Y_s|^{p-2}{\bf 1}_{\{Y_s\neq 0\}}|g(s,Y_s,Z_s)|^{2}ds\\
&&+p\int_{t}^{T}|Y_{s}|^{p-1}\langle \hat{Y}_{s},g(s,Y_{s},Z_s)\overleftarrow{dB}_{s}\rangle-p\int_{t}^{T}|Y_{s}|^{p-1}\langle \hat{Y}_{s},Z_{s}dW_{s}\rangle.
\end{eqnarray*}
\end{corollary}
\begin{proof}
The proof follows from Lemma 2.2. Indeed, recall that $(Y,Z)$ is solution of BDSDE \eqref{a1} and replace
$(X,K,H,G)$ by $(Y,f(.,Y,Z),Z,g(.,Y,Z))$, it follows that
\begin{eqnarray}
&&|Y_t |^{p}+\frac{p}{2}\int_{t}^{T}|Y_{s}|^{p-2}{\bf 1}_{\{Y_s\neq 0\}}\{(2-p)(|Z_s|^{2}-\langle \hat{Y}_{s},Z_sZ_s^{*}\hat{Y}_{s}\rangle)+(p-1)|Z_s|^{2}\}ds\nonumber\\
&=&|Y_{T}|^{p}+p\int_{t}^{T}|Y_{s}|^{p-1}\langle\hat{Y}_{s},f(s,Y_s,Z_s)\rangle ds+p\int_{t}^{T}|Y_{s}|^{p-1}\langle \hat{Y}_{s},g(s,Y_s,Z_s)\overleftarrow{dB}_{s}\rangle-p\int_{t}^{T}|Y_{s}|^{p-1}\langle \hat{Y}_{s},Z_{s}dW_{s}\rangle\nonumber\\
&&+\frac{p}{2}\int_{t}^{T}|Y_{s}|^{p-2}{\bf 1}_{\{Y_s\neq 0\}}\{(2-p)(|g(s,Y_s,Z_s)|^{2}-\langle \hat{Y}_{s},g(s,Y_s,Z_s)g^{*}(s,Y_s,Z_s)\hat{Y}_{s}\rangle)+(p-1)|g(s,Y_s,Z_s)|^{2}\}ds.\nonumber\\
\label{corollary1}
\end{eqnarray}
Since $p\in (1,2)$, it follows from \eqref{ingd} that
\begin{eqnarray}
&&(p-1)\int_{t}^{T}|Y_{s}|^{p-2}{\bf 1}_{\{Y_s\neq 0\}}|Z_s|^{2}ds\nonumber\\
&\leq& \int_{t}^{T}|Y_{s}|^{p-2}{\bf 1}_{\{Y_s\neq 0\}}\{(2-p)(|Z_s|^{2}-\langle \hat{Y}_{s},Z_sZ_s^{*}\hat{Y}_{s}\rangle)+(p-1)|Z_s|^{2}\}ds.\label{corollary2}
\end{eqnarray}
and
\begin{eqnarray}
&&\int_{t}^{T}|Y_{s}|^{p-2}{\bf 1}_{\{Y_s\neq0\}}\{(2-p)(|g(s,Y_s,Z_s)|^{2}-\langle\hat{Y}_{s},g(s,Y_s,Z_s)g^{*}(s,Y_s,Z_s)\hat{Y}_{s}\rangle)+(p-1)|g(s,Y_s,Z_s)|^{2}\}ds\nonumber\\
&\leq&(3-p)\int_{t}^{T}|Y_{s}|^{p-2}{\bf 1}_{\{Y_s\neq0\}}|g(s,Y_s,Z_s)|^{2}ds.\label{corollary3}
\end{eqnarray}
Therefore putting \eqref{corollary2} and \eqref{corollary3} to \eqref{corollary1} we obtain
\begin{eqnarray*}
&&|Y_t |^{p}+c(p)\int_{t}^{T}|Y_{s}|^{p-2}{\bf 1}_{\{Y_s\neq 0\}}|Z_s|^{2}ds\nonumber\\
&\leq&|Y_{T}|^{p}+p\int_{t}^{T}|Y_{s}|^{p-1}\langle\hat{Y}_{s},f(s,Y_s,Z_s)\rangle ds+p\int_{t}^{T}|Y_{s}|^{p-1}\langle \hat{Y}_{s},g(s,Y_s,Z_s)\overleftarrow{dB}_{s}\rangle-p\int_{t}^{T}|Y_{s}|^{p-1}\langle \hat{Y}_{s},Z_{s}dW_{s}\rangle\nonumber\\
&&+\bar{c}(p)\int_{t}^{T}|Y_{s}|^{p-2}{\bf 1}_{\{Y_s\neq 0\}}|g(s,Y_s,Z_s)|^{2}ds.
\label{corollary}
\end{eqnarray*}
which proved the result.
\end{proof}

\section{A priori estimates}
In this section, we state some estimations concerning solution to BDSDE \eqref{a1}. These estimates are very useful for
the study of existence and uniqueness of solution. In what follows, we are two difficulty. The function $f$ is not Lipschitz continuous and
we desire obtain estimates in $L^p$-sense, $p\in(1,2)$.

We begin by derive the following result which permit us to control the process
$Z$ by the data and the process $Y$.
\begin{lemma}
Let assumptions $\left({\bf H1}\right)$-$\left({\bf H2 }\right)$ hold and $\left( Y,Z\right) $ be a solution of BDSDE $(\ref{a1})$.
If $Y\in \mathcal{S}^{p}$ then $Z$ belong to $\mathcal{M}%
^{p}$ and there exists a real constant $C_{p,\lambda}$ depending only on $p,\; T$ and $\lambda$
such that
\begin{eqnarray*}
\E\left[ \left( \int_{0}^{T}|Z_{r}|^{2}dr\right)
^{p/2}\right]
  &\leq &C_{p}%
\E\left\{ \sup_{0\leq t\leq T}|Y_{t}|^{p}+\left(
\int_{0}^{T}|f^{0}_{r}|dr\right)^{p}+ \left( \int_{0}^{T}|g^{0}_{r}|^{2}dr\right)
^{p/2}\right\}.
\end{eqnarray*}
\end{lemma}
\begin{proof}
For each integer $n$, let us define
\begin{eqnarray*}
\tau_{n}=\inf \left \{t\in [0,T],\int_{0}^{t}|Z_{r}|^2dr \geq n \right \}
 \wedge T.
\end{eqnarray*}
The sequence $(\tau_n)_{n\geq 0}$ is stationary since the process $Z$ belongs to $L^{2}(0,T)$ and then\newline $\int^{T}_{0} |Z_s|^{2}ds <\infty$, $\P$- a.s.

For arbitrary real $a$, using It\^o's formula, we have
\begin{eqnarray}
&&|Y_{0}|^{2}+\int_{0}^{\tau_{n}}e^{ar}|Z_{r}|^{2}dr\nonumber\\ &=&
e^{a\tau_n}|Y_{\tau_{n}}|^{2}+2\int_{0}^{\tau_{n}}e^{ar}\langle Y_{r}, f(r,Y_{r},Z_{r})-aY_r\rangle dr
 + \int_{0}^{\tau_{n}} e^{ar}|g(r,Y_{r},Z_r)|^{2}dr\nonumber \\
&& + 2\int_{0}^{\tau_{n}}e^{ar}\langle Y_{r},g(r,Y_r,Z_r) \overleftarrow{dB}_{r}\rangle - 2\int_{0}^{\tau_{n}}
e^{ar}\langle Y_{r}, Z_{r} dW_{r}\rangle.\label{estZ1}
\end{eqnarray}
But, it follows from  assumptions $({\bf H1})$-$({\bf H2})$ and inequality $2bd\leq \frac{1}{\varepsilon}b^{2}+\varepsilon d^{2}$ that,
for any arbitrary positive real constant $\varepsilon$ and $\varepsilon'$,
\begin{eqnarray*}
2\langle Y_{r}, f(r,Y_{r},Z_{r})-aY_r\rangle&\leq& 2|Y_r||f^{0}_r|+2\mu|Y_r|^{2}+2\lambda|Y_r||Z_r|-a|Y_r|^{2}\\
&\leq &2|Y_r||f^{0}_r|+(2\mu+2\lambda+\varepsilon^{-1}\lambda^{2}-a)|Y_r|^{2}+\varepsilon|Z_r|^{2},
\\
\|g(r,Y_r,Y_r)\|^{2}&\leq& (1+\varepsilon')\lambda|Y_r|^{2}+(1+\varepsilon')\alpha|Z_r|^{2}+(1+\frac{1}{\varepsilon'})|g^{0}_r|^{2}.
\end{eqnarray*}
Thus, since $\tau_n\leq T$, taking $\varepsilon,\;\varepsilon'$ such that $\varepsilon+(1+\varepsilon')\alpha<1$ and\newline $2\mu+(3+\varepsilon')\lambda+\varepsilon^{-1}\lambda^{2}-a\leq 0$, we deduce
\begin{eqnarray}
\left(\int_{0}^{\tau_{n}}|Z_{r}|^{2}dr\right)^{p/2} &\leq&
C_{p,\lambda}\left\{\sup_{0\leq t\leq\tau_{n}}|Y_{t}|^{p}+\left(\int_{0}^{\tau_{n}}|f^{0}_{r}| dr\right)^{p}+\left(\int_{0}^{\tau_{n}}|g^{0}_{r}|^{2}dr\right)^{p/2}\right.\nonumber\\
&&\left.+\left|\int_{0}^{\tau_{n}}e^{a r}\langle Y_{r}, g(s,Y_r,Z_{r})\overleftarrow{dB}_{r}\rangle\right|^{p/2}+\left|\int_{0}^{\tau_{n}}e^{a r}\langle Y_{r}, Z_{r}dW_{r}\rangle\right|^{p/2}\right\}. \nonumber\\
\label{a2}
\end{eqnarray}

But thanks to BDG's inequality, we have
\begin{eqnarray*}
\E\left(\left|\int_{0}^{\tau_{n}}e^{a r}\langle Y_{r},Z_{r}dW_{r}\rangle
\right|^{p/2}\right)&\leq& d_{p}\E\left[ \left(\int_{0}^{\tau_{n}}
|Y_{r}|^{2}|Z_{r}|^{2}dr\right)^{p/4}\right]\nonumber \\
&\leq& \bar{C}_{p}\E\left[ \sup_{0\leq t\leq \tau_{n}}|Y_{t}|^{p/2}\left(\int_{0}^{\tau_{n}}
|Z_{r}|^{2}dr\right)^{p/4}\right]\nonumber\\
&\leq& \frac{\bar{C}_{p}^{2}}{\eta_1}\E\left(\sup_{0\leq t\leq \tau_{n}}|Y_{t}|^{p}\right)+
\eta_1\E\left(\int_{0}^{\tau_{n}}|Z_{r}|^{2}dr\right)^{p/2}.
\label{a3}
\end{eqnarray*}
and
\begin{eqnarray*}
\E\left(\left|\int_{0}^{\tau_{n}}e^{a r}\langle Y_{r}, g(s,Y_r,Z_{r})\overleftarrow{dB}_{r}\rangle\right|^{p/2}\right)&\leq& d_{p}\E\left[ \left(\int_{0}^{\tau_{n}}
|Y_{r}|^{2}|g(r,Y_r,Z_{r})|^{2}dr\right)^{p/4}\right]\nonumber \\
&\leq& \bar{C}_{p}\E\left[ \sup_{0\leq t\leq \tau_{n}}|Y_{t}|^{p/2}\left(\int_{0}^{\tau_{n}}
|g(r,Y_r,Z_{r})|^{2}dr\right)^{p/4}\right]\nonumber\\
&\leq& \frac{\bar{C}_{p}^{2}}{\eta_2}\E\left(\sup_{0\leq t\leq \tau_{n}}|Y_{t}|^{p}\right)+
\eta_2\E\left(\int_{0}^{\tau_{n}}|g(r,Y_r,Z_{r})|^{2}dr\right)^{p/2}\\
&\leq&C_p\E\left(\sup_{0\leq t\leq \tau_{n}}|Y_{t}|^{p}+\left(\int_{0}^{\tau_n}|g^{0}_r|^{2}\right)^{p/2}\right)\\
&&+(1+\eta_3)\eta_2\alpha\E\left(\int_{0}^{\tau_{n}}|Z_{r}|^{2}dr\right)^{p/2}.
\end{eqnarray*}
Let us take $\eta_1, \eta_2$ and $\eta_3$ small enough such that coming back to \eqref{a2}, we obtain, for each $n\in\N$,
\begin{eqnarray*}
\E\left[ \left( \int_{0}^{\tau_n}|Z_{r}|^{2}dr\right)
^{p/2}\right]
  &\leq &C_{p}%
\E\left\{ \sup_{0\leq t\leq T}|Y_{t}|^{p}+\left(
\int_{0}^{T}|f^{0}_{r}|dr\right)^{p}+ \left( \int_{0}^{T}|g^{0}_{r}|^{2}dr\right)
^{p/2}\right\},
\end{eqnarray*}
which by Fatou's lemma implies
\begin{eqnarray*}
\E\left[ \left( \int_{0}^{T}|Z_{r}|^{2}dr\right)
^{p/2}\right]
  &\leq &C_{p}%
\E\left\{ \sup_{0\leq t\leq T}|Y_{t}|^{p}+\left(
\int_{0}^{T}|f^{0}_{r}|dr\right)^{p}+ \left( \int_{0}^{T}|g^{0}_{r}|^{2}dr\right)
^{p/2}\right\},
\end{eqnarray*}
the desired result.
\end{proof}

We keep on this study by stating the estimate which is the main tool to derive existence and uniqueness result in our context. The difficulty
comes from the fact that $f$ is non-Lipschitz in $y$ and the function
$y\mapsto|y|^p$ is not $\mathcal{C}^{2}$ since we work with $p\in(1,2)$.
\begin{lemma}
Assume $\left(
\text{\textbf{H1}}\right)$-$\left({ \bf H2}\right)$. Let
$\left(Y,Z\right)$ be a solution of the backward doubly SDE
associated to the data $(\xi,f,g)$ where $Y$ belong to
$\mathcal{S}^{p}$. Then there exists a constant $C_{p,\lambda}$ depending
only on $p$ and $\lambda$ such that
\begin{eqnarray*}
\E\left\{\sup_{0\leq t \leq T}|Y_{t}|^{p}+
\left(\int_{0}^{T}|Z_{s}|^{2}ds\right)^{p/2}
\right\}&\leq& C_{p,\lambda}\E\left\{|\xi|^{p}+\left(\int_{0}^{T}|f^{0}_{s}|ds\right)^{p}\right.\\
&&+\left.\left(\int_{0}^{T}|g^{0}_{s}|^{2}ds\right)^{p/2}+\int_{0}^{T}|Y_s|^{p-2}{\bf 1}_{\{Y_s\neq 0\}}|g^{0}_{s}|^{2}ds\right\}.
\end{eqnarray*}
\end{lemma}
\begin{proof}
Applying Corollary 2.1 we have, for any $a>0$ and any $0\leq t\leq u\leq T$:
\begin{eqnarray*}
&&e^{apt}|Y_{t}|^{p}+
c(p)\int_{t}^{u}e^{aps}|Y_s|^{p-2}{\bf 1}_{\{Y_s\neq\ 0\}}|Z_{s}|^{2}ds\\
&\leq& e^{apu}|Y_u|^{p}-ap\int_{t}^{u}e^{aps}|Y_s|^{p}ds+p\int_{t}^{u}e^{aps}|Y_s|^{p-1}\langle\widehat{Y}_s,f(s,Y_s,Z_s)\rangle ds\nonumber\\
&&+\bar{c}(p)\int_{t}^{u}e^{aps}|Y_s|^{p-2}{\bf 1}_{\{Y_s\neq\ 0\}}|g(s,Y_s,Z_s)|^{2}ds+p\int_{t}^{u}e^{aps}|Y_s|^{p-1}\langle \widehat{Y}_s,g(s,Y_s,Z_s)\overleftarrow{dB}_s\rangle\nonumber\\
&&-p\int_{t}^{u}e^{aps}|Y_s|^{p-1}\langle\widehat{Y}_s,Z_sdW_s\rangle.
\end{eqnarray*}
The assumption on $f$ and $g$ yields
\begin{eqnarray}
\langle \hat{y},f(s,y,z)\rangle&\leq& |f_{s}^{0}| + \mu|y| + \lambda|z|\nonumber\\
|g(s,y,z)|^{2}&\leq& (1+\varepsilon)\lambda|y|^{2}+(1+\varepsilon)\alpha|z|^{2}+(1+\frac{1}{\varepsilon})|g^{0}_s|^{2},\nonumber
\end{eqnarray}
for any arbitrary $\varepsilon>0$. Therefore for all $t\in[0,u]$, we get with probability one:
\begin{eqnarray}
&&e^{apt}|Y_{t}|^{p}+
c(p)\int_{t}^{u}e^{aps}|Y_s|^{p-2}{\bf 1}_{\{Y_s\neq\ 0\}}|Z_{s}|^{2}ds\nonumber\\
&\leq& e^{apu}|Y_u|^{p}+[p(\mu-a)+\bar{c}(p)(1+\varepsilon)\lambda]\int_{t}^{u}e^{aps}|Y_s|^{p}ds\nonumber\\
&&+p\int_{t}^{u}e^{aps}|Y_s|^{p-1}|f^{0}_s|ds+\bar{c}(p)(1+\varepsilon^{-1})\int_{t}^{u}e^{aps}|Y_s|^{p-2}{\bf 1}_{\{Y_s\neq\ 0\}}|g_{s}^{0}|^{2}ds\nonumber\\
&&+\bar{c}(p)(1+\varepsilon)\alpha\int_{t}^{u}e^{aps}|Y_s|^{p-2}{\bf 1}_{\{Y_s\neq\ 0\}}|Z_s|^{2}ds+p\lambda\int_{t}^{u}e^{aps}|Y_s|^{p-1}|Z_s|ds
\nonumber\\
&&+p\int_{t}^{u}e^{aps}|Y_s|^{p-1}\langle \widehat{Y}_s,g(s,Y_s,Z_s)\overleftarrow{dB}_s\rangle-p\int_{t}^{u}e^{aps}|Y_s|^{p-1}\langle \widehat{Y}_s,Z_sdW_s\rangle \label{Col} .
\end{eqnarray}
We deduce from the previous inequality that, $\P$-a.s.,
\begin{eqnarray*}
\int_{0}^{T}e^{aps}|Y_s|^{p-2}{\bf 1}_{\{Y_s\neq\ 0\}}|Z_{s}|^{2}ds<\infty.
\end{eqnarray*}
Moreover, we have
\begin{eqnarray*}
p\lambda|Y_s|^{p-1}|Z_s|\leq \gamma^{-1}\frac{p\lambda^{2}}{2(p-1)}|Y_s|^{p}+\gamma c(p)|Y_s|^{p-1}{\bf 1}_{\{Y_s\neq 0\}}|Z_s|^{2},
\end{eqnarray*}
for any arbitrary $\gamma>0$.

Next for $\gamma$  and $\varepsilon$ small enough and $a$ large enough such that $\alpha'=[(1-\gamma)c(p)-(1+\varepsilon)\alpha\bar{c}(p)]>0$ and $c(a,\varepsilon,\gamma)=p[\mu+\frac{(3-p)(1+\varepsilon)}{2}\lambda+\frac{\gamma^{-1}}{2(p-1)}\lambda^{2}-a]\leq 0$, we have
\begin{eqnarray}
&&|Y_t|^p+\alpha'\int_{t}^{T}e^{aps}|Y_s|^{p-2}{\bf 1}_{\{Y_s\neq\ 0\}}|Z_{s}|^{2}ds\nonumber\\
&\leq& e^{apT}|\xi|^{p}+p\int_{t}^{T}e^{aps}|Y_s|^{p-1}|f^{0}_s|ds+\bar{c}(p)(1+\varepsilon^{-1})\int_{t}^{T}e^{aps}|Y_s|^{p-2}{\bf 1}_{\{Y_s\neq\ 0\}}|g_{s}^{0}|^{2}ds\nonumber\\
&&+p\int_{t}^{T}e^{aps}|Y_s|^{p-1}\langle \widehat{Y}_s,g(s,Y_s,Z_s)\overleftarrow{dB}_s\rangle
-p\int_{t}^{T}e^{aps}|Y_s|^{p-1}\langle \widehat{Y}_s,Z_sdW_s\rangle\nonumber\\
&\leq & X+p\int_{t}^{T}e^{aps}|Y_s|^{p-1}\langle \widehat{Y}_s,g(s,Y_s,Z_s)\overleftarrow{dB}_s\rangle-p\int_{t}^{T}e^{aps}|Y_s|^{p-1}\langle \widehat{Y}_s,Z_sdW_s\rangle,\label{estYZ3}
\end{eqnarray}
where
\begin{eqnarray*}
X &=&e^{apT}|\xi|^{p}+p\int_{0}^{T}e^{aps}|Y_s|^{p-1}|f^{0}_s|ds+c(p)(1+\varepsilon^{-1})\int_{0}^{T}e^{aps}|Y_s|^{p-2}{\bf 1}_{\{Y_s\neq\ 0\}}|g_{s}^{0}|^{2}ds.
\end{eqnarray*}
On can show that $M_t=\int_{t}^{T}e^{aps}|Y_s|^{p-1}\langle \widehat{Y}_s,g(s,Y_s,Z_s)\overleftarrow{dB}_s\rangle$
 and $N_t=\int_{t}^{T}e^{aps}|Y_s|^{p-1}\langle \widehat{Y}_s,Z_sdW_s\rangle$ are uniformly integrable martingale. Indeed, using BDG inequality and then Young's inequality we have,
\begin{eqnarray}
\E\langle M,M\rangle_T^{1/2}&\leq& \E\left[\sup_{0\leq t\leq T}|Y_t|^{p-1}\left(\int_{0}^{T}|Z_r|^{2}dr\right)^{1/2}\right]\nonumber\\
&\leq&\frac{p-1}{p}\E\left(\sup_{0\leq t\leq T}|Y_t|^{p}\right)+\frac{1}{p}\E\left[\left(\int_{0}^{T}|Z_r|^{2}dr\right)^{p/2}\right]\label{finite1}
\end{eqnarray}
and
\begin{eqnarray}
\E\langle N,N\rangle_T^{1/2}&\leq& \E\left[\sup_{0\leq t\leq T}|Y_t|^{p-1}\left(\int_{0}^{T}|g(r,Y_r,Z_r)|^{2}dr\right)^{1/2}\right]\nonumber\\
&\leq&\frac{p-1}{p}\E\left(\sup_{0\leq t\leq T}|Y_t|^{p}\right)+\frac{1}{p}\E\left[\left(\int_{0}^{T}|g(r,Y_r,Z_r)|^{2}dr\right)^{p/2}\right].\label{finite2}
\end{eqnarray}
The last term of \eqref{finite1} and \eqref{finite2} being finite since $Y$ and $g(.,Y,Z)$ belong to $\mathcal{S}^p$ and $\mathcal{M}^p$ respectively, and then $Z$ belongs to $\mathcal{M}^p$ by Lemma 3.1.

Return to \eqref{estYZ3}, we get both
\begin{eqnarray}
\E\left(e^{apt}|Y_{t}|^{p}\right)&\leq& \E(X),\nonumber\\
\alpha'\E\int_{0}^{T}e^{aps}|Y_s|^{p-2}{\bf 1}_{\{Y_s\neq\ 0\}}|Z_{s}|^{2}ds&\leq &C_p\E(X)\label{estZY1},\\
\E\left(\sup_{0\leq t\leq T}e^{apt}|Y_{t}|^{p}\right)&\leq& \E(X)+k_p\E\langle N,N\rangle_T^{1/2}+h_p\E\langle M,M\rangle_T^{1/2}.\nonumber
\end{eqnarray}
On the other hand, we also have
\begin{eqnarray*}
k_p\E\langle M,M\rangle_T^{1/2}&\leq&\frac{1}{4}\E\left(\sup_{0\leq t\leq T}e^{apt}|Y_{t}|^{p}\right)+4h^{2}_p\E\int_{0}^{T}e^{aps}|Y_s|^{p-2}{\bf 1}_{\{Y_s\neq\ 0\}}|Z_{s}|^{2}ds
\end{eqnarray*}
and
\begin{eqnarray*}
h_p\E\langle N,N\rangle_T^{1/2}&\leq&\frac{1}{4}\E\left(\sup_{0\leq t\leq T}e^{apt}|Y_{t}|^{p}\right)+4k^{2}_p\E\int_{0}^{T}e^{aps}|Y_s|^{p-2}{\bf 1}_{\{Y_s\neq\ 0\}}|Z_{s}|^{2}ds\\
&&+d_p\E\left(\int^T_0e^{aps}|Y_s|^{p}ds+\int_{0}^{T}e^{aps}|Y_s|^{p-2}{\bf 1}_{\{Y_s\neq\ 0\}}|g_{s}^{0}|^{2}ds\right).
\end{eqnarray*}
Therefore from \eqref{estZY1}, we obtain
\begin{eqnarray*}
\E\left(\sup_{0\leq t\leq T}e^{apt}|Y_{t}|^{p}\right)\leq C_p\E(X).
\end{eqnarray*}
Applying once again Young's inequality, we get

\begin{eqnarray*}
p C_p\int_{0}^{T}e^{aps}|Y_s|^{p-1}|f^{0}_s|ds&\leq &\frac{1}{2}\sup_{0\leq s\leq T}|Y_s|^{p}+C'_p\left(\int_{0}^{T}e^{aps}|f^{0}_s|ds\right)^{p}
\end{eqnarray*}
from which we deduce, in view of $X$, that
\begin{eqnarray*}
\E\left(\sup_{0\leq t\leq T}e^{apt}|Y_{t}|^{p}\right)\leq C_p\E\left[|\xi|^{p}+\left(\int_{0}^{T}e^{aps}|f^{0}_s|ds\right)^{p}
+\int_{0}^{T}e^{aps}|Y_s|^{p-2}{\bf 1}_{\{Y_s\neq\ 0\}}|g^{0}_{s}|^{2}ds\right].
\end{eqnarray*}
The result follows from Lemma 3.1.
\end{proof}

\section{Existence and uniqueness of a solution}
\setcounter{theorem}{0} \setcounter{equation}{0}
This section is devoted to derive existence and uniqueness result to BDSDE $(\xi,f,g)$ in
$L^{p}$-sense, ($p\in(1,2)$). Above a priori estimates and $L^{\infty}$-approximation is used under $({\bf H1})$-$({\bf H2})$ and the the following additional assumptions.
\begin{description}
\item $({\bf H3})$ For $p>1$,\\
$
\left\{
\begin{array}{l}
(i)\;\E\left[|\xi|^{p}\right]<\infty,\\\\
(ii)\;\P\, a.s.\; \forall\, (t,z)\in[0,T]\times\R^{d\times d},\; y\mapsto f(t,y,z)\, \mbox{is continuous},\\\\
(iii)\; g(.,0,0)\equiv 0,\\\\
(iv)\;\forall\, r>0, \; \psi_r(t)=\sup_{|y|<r}|f(t,y,0)-f_{t}^{0}|\in L^{1}([0,T]\time\Omega,dt\otimes\P).
\end{array}\right.
$
\end{description}

Firstly, we generalize the result of Pardoux and Peng (see Theorem 1.1, \cite{PP1}) to monotone case. To do this, let assume this assumption which appear in \cite{P}.
\begin{description}
\item ({\bf H4})\;$\P-\, a.s.\; \forall\, (t,y)\in[0,T]\times\R^{k},\, |f(t,y,0)|\leq |f(t,0,0)|+\varphi(|y|)$.
\end{description}

\begin{theorem}
Let $p=2$. Under assumptions $({\bf H1)}$-$({\bf H4)}$, BDSDE $(\ref{a1})$ has a
unique solution in $\mathcal{S}^{2}(\R^d) \times \mathcal{M}^2(\R^{d\times d})$.
\end{theorem}
\begin{proof}
It follows by combining argument of Pardoux (see Theorem 2.2 \cite{P}) with one used in Pardoux and Peng (see Theorem 1.1, \cite{PP1}).
Therefore, we will give the main line.\newline
{\bf Uniqueness}\newline Let $(Y,Z)$ and $(Y',Z')$ be two solutions of BDSDE $(\xi,f,g)$ verify above assumptions. It follows from Itô formula that
\begin{eqnarray*}
&&\E(|Y_t-Y'_t|^2)+\E\left(\int_{t}^{T}|Z_s-Z'_s|^2ds\right)\\&=&2\E\left(\int_{t}^{T}\langle Y_s-Y'_s,f(s,Y,Z_s)-f(s,Y'_s,Z'_s)\rangle ds\right)+\E\left(\int_{t}^{T}|g(s,Y_s,Z_s)-g(s,Y'_s,Z'_s)|^2ds\right)\\
&\leq& \E\left(\int_{t}^{T}[(2\mu+\lambda)|Y_s-Y'_s|^2+2\lambda|Y_s-Y'_s||Z_s-Z'_s|+\alpha|Z_s-Z'_s|^2]ds\right)\\
&\leq& (2\mu+\lambda+\gamma^{-1}\lambda^2)\E\left(\int_{t}^{T}|Y_s-Y'_s|^2ds\right)+(\alpha+\gamma)\E\left(\int_{t}^T|Z_s-Z'_s|^2 ds\right).
\end{eqnarray*}
Hence, taking $\gamma$ small enough such that $1-\alpha-\gamma>0$, we have
\begin{eqnarray*}
\E(|Y_t-Y'_t|^2)&\leq& C\E\left(\int_{t}^{T}|Y_s-Y'_s|^2ds\right),
\end{eqnarray*}
which provide with Gronwall's lemma that $\E(|Y_t-Y'_t|^2)=0,\,$ for all $t\in[0,T]$, and then $\displaystyle{\E\left(\int_{t}^{T}|Z_s-Z'_s|^2ds\right)=0}$.

{\bf Existence}\newline
Firstly, let state this result which is proved similarly as Proposition 2.4 (see \cite{P}) with additional computations due
to backward stochastic integral with respect Brownian motion $B$, so we omit it.
\begin{proposition}\label{Prop}
Given $V\in\mathcal{M}^2(0,T,\R^{d\times d})$ and assume $({\bf H1})$-$({\bf H4})$, there exists a unique measurable processes $(Y_t,Z_t)_{\{0\leq t\leq T\}}$
with values in $\R^{d}\times\R^{d\times d}$  satisfies
\begin{eqnarray*}
Y_t=\xi+\int_{t}^{T}f(s,Y_s,V_s)ds+\int_t^Tg(s,Y_s,Z_s)\overleftarrow{dB}_s-\int_t^TZ_sdW_s,\;\;0\leq t\leq T.
\end{eqnarray*}
\end{proposition}
Using Proposition \eqref{Prop}, we consider $\Phi:\mathcal{S}^2(\R^{d})\times\mathcal{M}^2(0,T,\R^{d\times d})\rightarrow
\mathcal{S}^2(\R^{d})\times\mathcal{M}^2(0,T,\R^{d\times d})$ defined by $(Y,Z)=\Phi(U,V)$ is the unique solution of the BDSDE
\begin{eqnarray*}
Y_t=\xi+\int_{t}^{T}f(s,Y_s,V_s)ds+\int_t^Tg(s,Y_s,Z_s)\overleftarrow{dB}_s-\int_t^TZ_sdW_s,\;\;0\leq t\leq T.
\end{eqnarray*}
Let $(U,V),\;(U',V')$ belong in $\mathcal{S}^2(\R^{d})\times\mathcal{M}^2(0,T,\R^{d\times d}),\; (Y,Z)=\Phi(U,V)$ and $(Y',Z')=\Phi(U',V')$. Setting
$(\overline{U},\overline{V})=(U-U',V-V')$ and $(\overline{Y},\overline{Z})=(Y-Y',Z-Z')$, it follows from Itô formula that for $\gamma\in\R$,
\begin{eqnarray*}
&&e^{\gamma t}\E|\overline{Y}_t|^2+\E\int_t^T
e^{\gamma s} (\gamma\left|\overline{Y}_s\right|^2 +|\overline{Z}_s|^2)ds\\
&=&2\E\int_t^T\langle \overline{Y}_s, f(s,Y_s,V_s)-f(s,Y',V')\rangle ds+\E\int_t^T |g(s,Y_s,Z_s)-g(s,Y',Z')|^2ds\\
&\leq&\E\int_t^T(2\mu+\frac{\lambda^2}{\varepsilon}+\lambda)|\overline{Y}_s|^2+\varepsilon|\overline{V}_s|^2+\alpha|\overline{Z}_s|^2)ds.
\end{eqnarray*}
Hence, if we choose $\gamma=\frac{\lambda^2}{\varepsilon}+\lambda+1-\alpha$, we have
\begin{eqnarray*}
\E\int_t^Te^{\gamma s} (\left|\overline{Y}_s\right|^2 +|\overline{Z}_s|^2)ds&\leq&
\frac{\varepsilon}{1-\alpha}\left(\E\int_t^T e^{\gamma s}(\left|\overline{U}_s\right|^2+\left|\overline{V}_s\right|^2)ds\right)
\end{eqnarray*}
Taken $\varepsilon<1-\alpha$, $\Phi$ is a strict contraction on
$\mathcal{S}^{2}(\R^d)\times \mathcal{M}^{2}((0,T);\R^{d\times d})$
equipped with the norm
\begin{eqnarray*}
\|(Y,Z)\|^{2}=\E\int_t^T
e^{\gamma s}\left( \left|Y_s\right|^2+|Z_s|^2\right)ds.
\end{eqnarray*}
Its unique fixed point solves BDSDE $(\xi,f,g)$ in $\mathcal{S}^{2}(\R^d) \times \mathcal{M}^2(\R^{d\times d})$.
\end{proof}

We are now ready to state our main result in this paper which is the existence and uniqueness of the solution to BDSDEs \eqref{a1} in $L^p$-sense.
\begin{theorem}
Let $p\in (1,2)$. Under  assumptions $({\bf H1)}$-$({\bf H3)}$, BDSDE $(\ref{a1})$ has a
unique solution in $\mathcal{S}^{p}(\R^d) \times \mathcal{M}^p(\R^{d\times d})$.
\end{theorem}
\begin{proof}
{\bf Uniqueness}
\newline
Let us consider $(Y,Z)$ and $(Y',Z')$ two solutions of BDSDE $\left(\xi,f,g \right)$. Setting
$\overline{Y}=Y - Y'$ and $\overline{Z}=Z - Z'$, then the process $(\overline{Y}, \overline{Z})$ solves BDSDE
\begin{eqnarray*}
\overline{Y}_t=\int_{t}^{T}\varphi(s,\overline{Y}_s,\overline{Z}_s)ds+\int_{t}^{T}\psi(s,\overline{Y}_s,\overline{Z}_s)\overleftarrow{dB}_s
-\int_{t}^{T}\overline{Z}_sdW_s,\;\; 0\leq t\leq T,
\end{eqnarray*}
where $$\varphi(s,y,z)=f(s,y+Y'_s,z+Z'_s)-f(s,Y'_s,Z'_s)\;\;\mbox{and}\;\;\psi(s,y,z)=g(s,y+Y'_s,z+Z'_s)-g(s,Y'_s,Z'_s).$$
Thanks to $({\bf H2})$, the process $(\overline{Y},\overline{Z})$ satisfies Lemma 3.2 with $\varphi^0=\psi^0=0$. Thus, $(\overline{Y},\overline{Z})=
(0, 0)$ immediately.

{\bf Existence}\newline It split in two steps.\newline
{\bf Step 1.} For positive real $r$, we suppose $\xi,\, \sup |f^{0}_{t}|$ are bounded random variables such that
\begin{eqnarray*}
e^{(1+\lambda^{2})T}(|\xi|+T\|f^{0}
\|_{\infty}) &<& r.
\end{eqnarray*}
For such $r$, we define $\theta_{r}$ a smooth function such that $0\leq
\theta_{r}\leq 1$ and
\begin{eqnarray*}
\theta_{r}(y)=\left\{ \begin{array}{l}
1\text{ for }|y|\leq r \\
\\
0\text{ for }|y|\geq r+1.
\end{array}
\right.
\end{eqnarray*}
For each $n\in\N^{*}$, we denote
\begin{eqnarray*}
h_{n}(t,y,z)&=&\theta_{r}(y)(f(t,y,q_{n}(z))-f_{t}^{0})\frac{n}{\pi_{r+1}
(t)\vee n}+f_{t}^{0},
\end{eqnarray*}
where  $q_{n}(z)=z\frac{n}{|z|\vee n }$.

For each $n\in\N$, as in \cite{Pal}, $(\xi,h_{n},g)$ satisfies assumptions $({\bf H1)}$-$({\bf H4)}$ with $\mu$ positive. Then BDSDE $(\xi,h_n,g)$ has a unique process solution $(Y^n, Z^n)$ in the space $\mathcal{S}^{2}(\R^d) \times \mathcal{M}^2(\R^{d\times d})$. Moreover, it follows from  \cite{BC} (see, Lemma 2.2) and \cite{Bal} (see, remark 2.2) that $\|Y^{n}\|_{\infty}\leq r$. In addition, using Lemma 3.1,
\begin{eqnarray}
\|Z^{n}\|_{\mathcal{M}^{2}}\leq r',\label{anotherest}
\end{eqnarray}
where $r'$ is another constant depending on $r$.

As a byproduct $(Y^{n},Z^{n})$ is solution to BDSDE $(\xi,f_{n},g)$ where
\begin{eqnarray*}
f_{n}(t,y,z)=(f(t,y,q_{n}(z))-f_{t}^{0})\frac{n}{\pi_{r+1}
(t)\vee n}+f_{t}^{0}.
\end{eqnarray*}
For each $i\in\N$, let set
$\begin{array}{l}
\bar{Y}^{n,i}=Y^{n+i}-Y^{n},\ \bar{Z}^{n,i}=Z^{n+i}-Z^{n}.\end{array}
$
Then, we have, using assumptions $({\bf H2})$ on $f_{n+i}$
\begin{eqnarray*}
&&e^{at}|\bar{Y}_{t}^{n,i}|^{2}+(1-\varepsilon-\alpha)\int_{t}^{T}
e^{a s}|\bar{Z}_{s}^{n,i}|^{2}ds \\
&\leq&
2\int_{t}^{T}e^{a s}\langle\bar{Y}_{s}^{n,i},f_{n+i}(s,Y_{s}^{n},Z_{s}^{n})
-f_{n}(s,Y_{s}^{n},Z_{s}^{n})\rangle ds \\
&&+(2\mu+\frac{1}{\varepsilon}\lambda^{2}+\lambda-a)\int_{t}^{T}e^{a s}|\bar{Y}^{n,i}_{s}|^{2}ds\\
&&+2\int_{t}^{T}e^{a s}\langle\bar{Y}_{s}^{n,i},(g(s,Y^{n+i}_s,Z^{n+i}_s)-g(s,Y^{n}_s,Z^{n}_s))dB_{s}\rangle\\
&&-2\int_{t}^{T}e^{a s}\langle\bar{Y}_{s}^{n,i},\bar{Z}_{s}^{n,i}dW_{s}\rangle,
\end{eqnarray*}
for any $a>0$ and $\varepsilon>0$.

Next, since $\|\bar{Y}^{n,i}\|_{\infty}\leq 2r$ and setting $\gamma=1-\varepsilon-\alpha>0$ and
$(2\mu+\frac{1}{\varepsilon}\lambda^{2}+\lambda-a)\leq 0$, we obtain
\begin{eqnarray*}
&&e^{at}|\bar{Y}_{t}^{n,i}|^{2}+\gamma\int_{t}^{T}
e^{a s}|\bar{Z}_{s}^{n,i}|^{2}ds \\
&\leq&
4r\int_{t}^{T}e^{a s}|f_{n+i}(s,Y_{s}^{n},Z_{s}^{n})
-f_{n}(s,Y_{s}^{n},Z_{s}^{n})|ds \\
&&+2\int_{t}^{T}e^{a s}\langle\bar{Y}_{s}^{n,i},(g(s,Y^{n+i}_s,Z^{n+i}_s)-g(s,Y^{n}_s,Z^{n}_s))dB_{s}\rangle\\
&&-2\int_{t}^{T}e^{a s}\langle\bar{Y}_{s}^{n,i},\bar{Z}_{s}^{n,i}dW_{s}\rangle.
\end{eqnarray*}
Therefore, combining rigorously Gronwall's and BDG inequality, there exists a constant $C$ depending only on $\lambda,\,\alpha$ and $T$ such that
\begin{eqnarray*}
&&\E\left[\sup_{0\leq t\leq T}|\bar{Y}_{t}^{n,i}|^{2}+\int_{0}^{T}|\bar{Z}_{s}^{n,i}|^{2}ds\right]\leq
Cr\E\left[\int_{0}^{T}|f_{n+i}(s,Y_{s}^{n},Z_{s}^{n})
-f_{n}(s,Y_{s}^{n},Z_{s}^{n})|ds\right].
\end{eqnarray*}
On the other hand, since $\|Y^{n}\|_{\infty}\leq r$,
\begin{eqnarray*}
|f_{n+i}(s,Y_{s}^{n},Z_{s}^{n})-f_{n}(s,Y_{s}^{n},Z_{s}^{n})|
&\leq& 2\lambda|Z^{n}_s|{\bf 1}_{\{|Z^{n}_s|\ >n\}}
+2\lambda|Z^{n}_s|{\bf
1}_{\{\pi_{r+1}(s)>n\}}+2\pi_{r+1}(s){\bf 1}_{\{\pi_{r+1}(s)>n\}},
\end{eqnarray*}
from which, we deduce with the help of inequality \eqref{anotherest} and assumption $({\bf H3}iv)$, that $(Y^{n},Z^{n})$ is a cauchy sequence in Banach space $\mathcal{S}^{2}(\R^d)\times\mathcal{M}^2(\R^{d\times d})$. Hence,
$(Y^{n},Z^{n})$ admit a limit $(Y,Z)\in\mathcal{S}^{2}(\R^d)\times\mathcal{M}^2(\R^{d\times d}) $, which solves BDSDE \eqref{a1}.

{\bf Step 2.}\ In the general case, let us define for each $n\in \N^{*}$,
\begin{eqnarray*}
\xi_{n}= q_{n}(\xi),\,\;\;f_{n}(t,y,z)= f\left(t,y,z\right)-f^{0}_{t}+q_{n}(f^{0}_{t}).
\end{eqnarray*}
Thanks to the Step 1, BDSDE $(\xi_{n},f_{n},g)$ has a unique solution $(Y^{n},Z^{n})\in L^{2}$, but also in $L^{p}$ far all $p\in (1,2)$ according to Lemma 3.1.
Moreover, from Lemma 3.2, for $(i,n)\in \N\times\N^{*}$, there exists $C(T,\alpha,\lambda)$ such that
\begin{eqnarray}
&&\E\left\{\sup_{0\leq t\leq T}|Y_{t}^{n+i}-Y_{t}^{n}|^{p}+
\left(\int_{0}^{T}|Z_{s}^{n+i}-Z_{s}^{n}|^{2}ds\right)^{p/2}\right\}\nonumber\\
&\leq& C\E\left\{|\xi_{n+i}-\xi_{n}|^{p}+\left(\int_{0}^{T}
|q_{n+i}(f^{0}_{s})-q_{n}(f^{0}_{s})|ds\right)^{p}\right\}.\label{last}
\end{eqnarray}

The right-hand side of \eqref{last} tends to $0$, as $n\rightarrow +\infty$, uniformly on $i$, so $(Y^{n},Z^{n})$ is again a Cauchy sequence in
$\mathcal{S}^{p}(\R^d)\times\mathcal{M}^p(\R^{d\times d})$ and its limit $(Y,Z)$ solves BDSDE $(\ref{a1})$.
\end{proof}

{\bf Acknowledgments}\\
The author thanks the anonymous referee for his careful reading to improve the manuscript

\end{document}